\documentclass[11pt, reqno]{amsart}
\usepackage{amssymb, amstext, amscd, amsmath, amsthm}
\usepackage{color}
\usepackage{hyperref} 

\usepackage{xy}
\xyoption{all}

\theoremstyle{plain}
\newtheorem{theorem}{Theorem}[section]
\newtheorem{thm}[theorem]{Theorem}
\newtheorem{cor}[theorem]{Corollary}
\newtheorem{prop}[theorem]{Proposition}
\newtheorem{lem}[theorem]{Lemma}
\theoremstyle{definition}
\newtheorem{rem}[theorem]{Remark}

\newtheorem{defn}[theorem]{Definition}

\newcommand{\bN}{{\mathbb{N}}}

\newcommand{\bT}{{\mathbb{T}}}

\newcommand{\bZ}{{\mathbb{Z}}}

  \newcommand{\A}{{\mathcal{A}}}
  \newcommand{\B}{{\mathcal{B}}}

  \newcommand{\E}{{\mathcal{E}}}
  
  \newcommand{\G}{{\mathcal{G}}}

\newcommand{\M}{{\mathcal{M}}}

\newcommand{\fA}{{\mathfrak{A}}}

\newcommand{\fD}{{\mathfrak{D}}}

\newcommand{\fF}{{\mathfrak{F}}}

\newcommand{\rC}{{\mathrm{C}}}

\newcommand{\upchi}{{\raise.35ex\hbox{\ensuremath{\chi}}}}

\newcommand{\qforal}{\quad\text{for all}\quad}

\newcommand{\ran}{\operatorname{Ran}}
\newcommand{\spn}{\operatorname{span}}

\newcommand{\Per}{\operatorname{Per}}

\newcommand{\ca}{\mathrm{C}^*}

\newcommand{\ol}{\overline}

\begin{document}
\title[Cycline subalgebras]
{Cycline subalgebras of $k$-graph C*-algebras}
\author[D. Yang]{Dilian Yang}
\address{Dilian Yang,
Department of Mathematics $\&$ Statistics, University of Windsor, 
401 Sunset Avenue, Windsor, ON
N9B 3P4, CANADA}
\email{dyang@uwindsor.ca}

\begin{abstract}
In this paper, we prove that the cycline subalgbra of a $k$-graph C*-algebra is 
maximal abelian, and show when it is a Cartan subalgebra (in the sense of Renault). 
\end{abstract}

\thanks{The author partially supported by an NSERC grant.}

\date{}
\maketitle

\section{Introduction}
\label{S:Intr}

Higher rank graph algebras (or $k$-graph algebras) have been attracting a lot of attention recently.
See, for example, \cite{Rae05} and the references therein. 
They were first introduced by Kumjian-Pask in 2000 \cite{KumPask} in order to generalize directed graph algebras
and higher rank Cuntz-Krieger algebras studied by Robertson and Steger \cite{RS99}. 
For a given $k$-graph $\Lambda$, its graph C*-algebra $\ca(\Lambda)$ is the universal C*-algebra
among Cuntz-Krieger $\Lambda$-families. 

One of the most important and central topics on $k$-graph algebras is to determine when a given representation $\pi$ from $\ca(\Lambda)$ to a C*-algebra $\A$ is injective. 
This is closely related to so called ``uniqueness theorems" in the literature. There are two such 
theorems: the \textit{gauge invariant uniqueness theorem} (GIUT) and the \textit{Cuntz-Krieger uniqueness theorem} (CKUT), 
which have been known for some time.
Both GIUT and CKUT conclude that $\pi$ is injective if and only if its restriction 
$\pi|_{\fD_\Lambda}$ of $\pi$ onto the diagonal algebra $\fD_\Lambda$ of $\ca(\Lambda)$ is injective,
under the following conditions: the GIUT requires the existence of an action 
$\theta$ of $\bT^k$ on $\A$ such that $\pi$ is equivariant between $\theta$ and the canonical gauge action $\gamma$ of $\bT^k$ on $\ca(\Lambda)$,
while the CKUT requires that  $\Lambda$ is aperiodic.

It is well-known that 
 the aperiodicity is a very stringent condition, 
and that it is very hard to check (even in single-vertex $2$-graphs \cite{DY09a}). Thus, a very important and necessary task is to find a more general version of 
the CKUT. This has been successfully achieved by Brown-Nagy-Reznikoff recently in \cite{BNR14}. The most natural candidate 
of $\fD_\Lambda$ in the general case is the so called \textit{cycline subalgebra} $\M_\Lambda$ (whose definition will be precisely given later). Brown-Nagy-Reznikoff 
proved the following \textit{generalized CKUT}: For \textit{any} row-finite source-free $k$-graph $\Lambda$, 
a representation $\pi$ of $\ca(\Lambda)$ is injective if and only if the restriction $\pi|_{\M_\Lambda}$ is injective.

Returning to an aperiodic $k$-graph $\Lambda$, it is known that $\fD_\Lambda$ is a MASA (maximal abelian subalgebra) in $\ca(\Lambda)$,
and that  there is a faithful conditional expectation from $\ca(\Lambda)$ onto $\fD_\Lambda$. 
Besides, in general, for a given abelian C*-subalgebra $\B$ of a C*-algebra $\A$, it is always nice and interesting to know if $\B$ is a MASA in $\A$ and 
if there is a faithful conditional expectation from $\A$ onto $\B$. 
So Brown-Nagy-Reznikoff asked the following two natural questions on $\M_\Lambda$ (cf. P.~2591 and P.~2601 in \cite{BNR14}):
\smallskip
\begin{center}
\begin{itemize}
\item[Q1.] \textit{Is $\M_\Lambda$ a MASA in $\ca(\Lambda)$?}
\item[Q2.] \textit{Is there a faithful conditional expectation from $\ca(\Lambda)$ onto $\M_\Lambda$?}
\end{itemize}
\end{center}
Our goal in this note is the following: (a) we answer Q1 affirmatively, 
(b) we study when Q2 has a positive answer, and so provide a condition which guarantees 
$\M_\Lambda$ is a Cartan subalgebra in $\ca(\Lambda)$ (in the sense of Renault \cite{Ren08}).

The remaining of this paper is organized as follows. In the next section, some necessary background is given.  
Q1 above is completely answered in Section \ref{S:Lemmas}.  In Section \ref{S:Cartan} we study when $\M_\Lambda$
is Cartan, and so in particular answer Q2.

\section{Preliminaries}
\label{S:Pre}

In this section, we give some necessary background which will be used later.  At the same time, we fix our notation. See \cite{BNR14, CKSS14, KumPask, Rae05, Yan14} for more information.

\subsection{$k$-graphs}

Let $k\ge 1$ be a natural number. Regard $\bN^k$ (containing $0$) as a small category with one object, and denote its standard generators 
as $e_1,\ldots, e_k$. 
A \textit{$k$-graph} (also known as \textit{rank $k$ graph}, or \textit{higher rank graph}) is a countable small category $\Lambda$ with a functor $d:\Lambda\to \bN^k$ satisfying the following 
factorization property:
Whenever $\xi\in\Lambda$ satisfies $d(\xi)=m+n$, there are unique elements $\eta, \zeta\in\Lambda$ such that 
$d(\eta)=m$, $d(\zeta)=n$ and $\xi=\eta\zeta$. 
For $n\in \bN^k$, let $\Lambda^n=d^{-1}(n)$, and so $\Lambda^0$ is the vertex set of $\Lambda$. 
There are source and range maps $s, r: \Lambda\to \Lambda^0$ such that $r(\xi)\xi s(\xi)=\xi$ for all $\xi\in\Lambda$.
For $v\in\Lambda^0$,
$v\Lambda=\{\xi\in\Lambda: r(\xi)=v\}$. 
We say that a $k$-graph $\Lambda$ is \textit{row-finite and source-free} if 
$0<|v\Lambda^n|<\infty$ for all $v\in\Lambda^0$ and $n\in \bN^k$. 

Let
\[
\Omega_k=\big\{(m,n)\in \bN^k\times \bN^k\mid m\le n\big\}. 
\]
Define $d,s,r:\Omega_k\to \bN^k$ by $d(m,n)=n-m$, $s(m,n)=n$, and $r(m,n)=m$.
One can check that $\Omega_k$ is a row-finite and source-free $k$-graph.

Let $\Lambda$ and $\Gamma$ be two $k$-graphs. A \textit{$k$-graph morphism} between $\Lambda$ and $\Gamma$ is a 
functor $x:\Lambda\to \Gamma$ such that $d_\Gamma(x(\lambda))=d_\Lambda(\lambda)$ for all $\lambda\in\Lambda$.  
The \textit{infinite path space of $\Lambda$} is defined as
\[
\Lambda^\infty=\big\{x:  \Omega_k\to \Lambda\mid x \text{ is a } k\text{-graph morphism}\big\}.
\]
If $\Lambda$ is row-finite and source-free, it is often useful to think of every element of $\Lambda^\infty$ as 
an infinite path, which contains infinitely many edges of degree $e_i$ for each $i\in\{1,\ldots, k\}$. 
For $x\in \Lambda^\infty$ and $n\in \bN^k$, there is a unique element $\sigma^n(x)\in \Lambda^\infty$ defined by
\[
\sigma^n(x)(q,r)=x(n+q,n+r).
\]
That is, $\sigma^n$ is a shift map on $\Lambda^\infty$. If $\mu\in\Lambda$ and $x\in s(\mu)\Lambda^\infty$, then $\mu x$ 
is defined to be the unique infinite path such that 
$\mu x(0,n)=\mu\cdot x(0,n-d(\mu))$ for any $n\in \bN^k$ with $n\ge d(\mu)$. 
If $\sigma^m(x)=\sigma^n(x)$ for some $m\ne n$ in $\bN^k$, $x$ is said to be (eventually) \textit{periodic}.

\begin{defn}
A $k$-graph $\Lambda$ is said to be \textit{periodic} if there is $v\in\Lambda^0$ such that every $x\in v\Lambda^\infty$ is periodic. 
Otherwise, $\Lambda$ is called \textit{aperiodic}. 
\end{defn}

\subsection{$k$-graph C*-algebras}

For a given row-finite and source-free $k$-graph $\Lambda$, we associate to it a universal C*-algebra $\ca(\Lambda)$ as follows.

\begin{defn}
Let $\Lambda$ be a row-finite and source-free $k$-graph. A \textit{Cuntz-Krieger $\Lambda$-family} in a C*-algebra $\A$ is a family $\{S_\lambda:\lambda\in\Lambda\}$ in $\A$
such that 
\begin{itemize}
\item[(CK1)] $\{S_v\mid v\in\Lambda^0\}$ is a set of mutually orthogonal projections;
\item[(CK2)] $S_\mu S_\nu=S_{\mu\nu}$ whenever $s(\mu)=r(\nu)$;
\item[(CK3)] $S_\mu^* S_\nu=\delta_{\mu,\nu}\,S_{s(\mu)}$ for all $\mu,\nu\in\Lambda$ with $d(\mu)=d(\nu)$; 
\item[(CK4)] $S_v=\sum_{\lambda\in v\Lambda^n}S_\lambda S_\lambda^*$ for all $v\in\Lambda^0$ and $n\in \bN^k$.
\end{itemize}
The \textit{$k$-graph C*-algebra $\ca(\Lambda)$} is the universal C*-algebra among Cuntz-Krieger $\Lambda$-families. 
In this paper, we use $\{s_\mu\mid \mu\in\Lambda\}$ to denote the universal Cuntz-Krieger $\Lambda$-family of $\ca(\Lambda)$.
\end{defn}

It is known that
\[
\ca(\Lambda)=\ol\spn\{s_\mu s_\nu^*: \mu, \nu\in \Lambda\}.
\]
By the universal property of $\ca(\Lambda)$, there is a natural gauge action 
$\gamma$ of $\bT^k$ on $\ca(\Lambda)$ defined by
\[
\gamma_t(s_\lambda)=t^{d(\lambda)}s_\lambda\qforal t\in\bT^k,\ \lambda\in\Lambda.
\]
Here $t^n=t_1^{n_1}\cdots t_k^{n_k}$ for all $n=(n_1,\ldots, n_k)\in \bZ^k$.
Averaging over $\gamma$ gives a faithful conditional expectation $\Phi$ from $\ca(\Lambda)$ onto the fixed point algebra $\ca(\Lambda)^\gamma$,
known as the \textit{core} of $\ca(\Lambda)$. 
It turns out that  $\ca(\Lambda)^\gamma$ is an AF algebra and 
\[
\fF_\Lambda:=\ca(\Lambda)^\gamma=\ol\spn\{s_\mu s_\nu^*: d(\mu)=d(\nu)\}.
\]
For sake of simplicity, put
\[
P_\mu:=s_\mu s_\mu^*\qforal \mu\in\Lambda.
\]
The \textit{diagonal algebra} $\fD_\Lambda$ of $\ca(\Lambda)$ is defined as 
\[
\fD_\Lambda=\ol\spn\{s_\mu s_\mu^*: \mu\in\Lambda\}=\ol\spn\{P_\mu: \mu\in\Lambda\},
\]
which is a MASA in $\fF_\Lambda$, but, generally not a MASA in $\ca(\Lambda)$. 

For each $n=(n_1,\ldots, n_k)\in\bZ^k$, define a mapping $\Phi_n$ on $\ca(\Lambda)$ via
\[
\Phi_n(x)=\int_{\bT^k}t^{-n}\gamma_t(x) dt \quad \mbox{for all}\quad x\in\ca(\Lambda).
\]
Then $\Phi_n$ acts on the standard generators via
\[
\Phi_n(s_\mu s_\nu^*)=
\left\{
\begin{array}{ll}
s_\mu s_\nu^*, & \quad \mbox{if}\quad d(\mu)-d(\nu)=n,\\
0, &\quad \mbox{otherwise}.
\end{array}
\right.
\]
So $\fF_\Lambda$ coincides with $\ran\Phi_0$, and $\ran\Phi_n$ is
spanned by the standard generators in $\ca(\Lambda)$ of ``degree $n$".
Also, as directed graph algebras \cite{HPP05}, every $x\in\ca(\Lambda)$ has a (unique) formal series
\[
x\sim \sum_{n\in\bZ^k}\Phi_n(x),
\]
which is Abel  summable (refer to \cite{Taylo} for information on Abel summable).  
It is often useful heuristically to work directly with the series of $x$.

Row-finite and source-free $k$-graph C*-algebras can be also constructed via second countable, \'etale locally compact groupoids 
\[
\G_\Lambda=\big\{(x,m-n,y)\in \Lambda^\infty\times \bZ^k\times \Lambda^\infty: \sigma^m(x)=\sigma^n(y)\big\}
\]
(cf. \cite{KumPask}). 
The following facts are well-known: A basis for the topology of $\G_\Lambda$ is given by the open compact cylinder sets 
\[
Z(\alpha,\beta)=\big\{(\alpha x, d(\alpha)-d(\beta), \beta x): x\in s(\alpha)\Lambda^\infty\big\},
\]
where $\alpha,\beta\in\Lambda$ with $s(\alpha)=s(\beta)$;
$\G_\Lambda$ is amenable; and  $\ca(\Lambda)\cong \ca(\G_\Lambda)\cong\ca_r(\G_\Lambda)$.
From \cite{Ren80}, $\ca(\G_\Lambda)$ consists of some elements of $\rC_0(\G_\Lambda)$,
the continuous functions on $\G_\Lambda$ vanishing at infinity. But also notice that $\ca(\G_\Lambda)$ contains $\rC_c(\G_\Lambda)$, 
the continuous functions on $\G_\Lambda$ with compact support.

\subsection{Periodicity}
\label{SS:per}

Let $\Lambda$ be a row-finite and source-free $k$-graph. 
Define an equivalence relation $\sim$ on $\Lambda$ as follows:
\begin{align}
\label{D:sim}
\mu\sim\nu\Longleftrightarrow s(\mu)=s(\nu) \text{ and } \mu x=\nu x \text{ for all }x\in s(\mu)\Lambda^\infty.
\end{align}
If $\mu\sim\nu$, obviously one also has $r(\mu)=r(\nu)$ automatically. So $\sim$ respects sources and ranges. 

Associate to the equivalence relation $\sim$ an important set -- the \textit{periodicity $\Per\Lambda$ of $\Lambda$}:
\[
\Per\Lambda=\big\{d(\mu)-d(\nu): \xi,\eta\in\Lambda, \ \mu\sim\nu\big\}\subseteq\bZ^k.
\]
In general, $\Per\Lambda$ is a subset of $\bZ^k$ containing $0$. 
Furthermore, $\Lambda$ is aperiodic if and only if $\Per\Lambda=\{0\}$ (cf., e.g., \cite{Yan14}).

The subalgebra we are particularly interested in here is 
\[
\M_\Lambda=\ca(s_\mu s_\nu^*: \mu\sim \nu \in \Lambda),
\]
which plays a vital role in this paper. 
$\M_\Lambda$ is called the \textit{cycline subalgebra of $\ca(\Lambda)$} in \cite{BNR14}, since it is related to generalized cycles introduced in \cite{ES12}. 
Actually, it is defined in terms of cycline pairs in \cite{BNR14}. But it is the same as the one defined above by \cite[Proposition 4.1]{BNR14} and the definition of $\sim$ in \eqref{D:sim}.
Clearly, $\M_\Lambda$ contains $\fD_\Lambda$. Furthermore, $\M_\Lambda=\fD_\Lambda$ if and only if $\Lambda$ is 
aperiodic by the characterization of aperiodicity mentioned above. 
It is also shown in \cite{BNR14} that the relative commutant $\fD_\Lambda'$ of $\fD_\Lambda$ in $\ca(\Lambda)$ is abelian, and
that $\M_\Lambda\subseteq\M_\Lambda' = \fD'_\Lambda$.

Let us finish off this section by the following conventions. 

\smallskip

\noindent
\textbf{Conventions.}
Throughout the rest of this paper, 
\begin{center}
\textbf{all $k$-graphs are assumed to be row-finite and source-free} 
\end{center}
without any further mention. 

For simplicity, we write $\fD_\Lambda'$ to really mean the relative commutant 
\[
\fD_\Lambda'=\big\{A\in\ca(\Lambda): AD=DA\text{ for all } D\in \fD_\Lambda\big\}.
\]
 

\section{Cycline subalgebras are MASA}

\label{S:Lemmas}

Let $\Lambda$ be a $k$-graph, and $\M_\Lambda$ be the cycline subalgebra of $\ca(\Lambda)$. Recall from \cite{BNR14} that 
\[
\M_\Lambda=\ca(s_\mu s_\nu^*: \mu\sim \nu)=\ol\spn\big\{s_\mu s_\nu^*: \mu\sim \nu\big\}.
\]

Our main goal in this section is to prove $\M_\Lambda=\fD_\Lambda'$, which in particular affirmatively answers 
Q1 mentioned in Introduction: $\M_\Lambda$ is always
a MASA.  But four auxiliary lemmas are needed. 
The first one is directly from \cite{BNR14}. 

\begin{lem}
\label{L:equchar}
\cite[Proposition 4.1]{BNR14}
Let $\mu,\nu\in\Lambda$ with $s(\mu)=s(\nu)$. The following are equivalent: 
\begin{enumerate}
\item $s_\mu s_\nu^*$ is normal and commutes with $\fD_\Lambda$;
\item $\mu\sim \nu$.
\end{enumerate}
\end{lem}

By Lemma \ref{L:equchar}, one has $\M_\Lambda\subseteq\fD_\Lambda'$. So in order to prove $\M_\Lambda=\fD_\Lambda'$, 
it is sufficient to verify $\fD_\Lambda'\subseteq\M_\Lambda$. 
Our first step is to prove that the standard generators in $\fD_\Lambda'$ belong to $\M_\Lambda$.
The following gives a strengthened version of Lemma \ref{L:equchar}, which will be very useful in what follows. 

\begin{lem}
\label{L:normal}
Let $\mu,\nu\in\Lambda$ with $s(\mu)=s(\nu)$. Then 
\[
s_\mu s_\nu^*\in\fD_\Lambda'\Longleftrightarrow \mu\sim\nu.
\]
\end{lem}

\begin{proof}
By Lemma \ref{L:equchar}, it is enough to prove that $s_\mu s_\nu^*\in \fD_\Lambda'$ implies that $s_\mu s_\nu^*$ is automatically normal, namely, 
\[
s_\mu s_\mu^*=s_\nu s_\nu^*, \text{ i.e., } P_\mu=P_\nu.
\]

Assume now that $\mu,\nu\in\Lambda$ with $s(\mu)=s(\nu)$ satisfies $s_\mu s_\nu^*\in\fD_\Lambda'$. Then one has the following implications: 
\begin{align*}
&s_\mu s_\mu^*\cdot s_\mu s_\nu^*=s_\mu s_\nu^*\cdot s_\mu s_\mu^*\\
\Rightarrow\ & s_\mu s_\nu^*=s_\mu s_\nu^*s_\mu s_\mu^*\\
\Rightarrow\ & s_\mu^* s_\mu s_\nu^*= s_\mu^* s_\mu s_\nu^*s_\mu s_\mu^*\\
\Rightarrow\ & s_{s(\mu)} s_\nu^*=s_{s(\mu)} s_\nu^* P_\mu\\
\Rightarrow\ & s_\nu^*=s_\nu^*P_\mu\\
\Rightarrow\ & s_\nu s_\nu^*=s_\nu s_\nu^*P_\mu\\
\Rightarrow\ & P_\nu= P_\nu P_\mu. 
\end{align*}

Clearly $s_\mu s_\nu^*\in \fD_\Lambda'$ implies $s_\nu s_\mu^*\in\fD_\Lambda'$. So switching $\mu$ and $\nu$ in the above process gives 
$P_\mu= P_\mu P_\nu$. Hence $P_\mu= P_\nu$ as $P_\mu P_\nu=P_\nu P_\mu$. This ends our proof.
\end{proof}

Roughly speaking, the next lemma says that, for our purpose, it is enough to consider the elements of $\fD_\Lambda'$ of degree $n\in\bZ^k$.

\begin{lem}
\label{L:Phin}
Let $A\in \ca(\Lambda)$. Then $A\in \fD_\Lambda'$ if and only if $\Phi_n(A)\in \fD_\Lambda'$ for all $n\in \bZ^k$.
\end{lem}

\begin{proof}
It suffices to show the ``only if" part. Let $D\in \fD_\Lambda$. Then for all $n\in\bZ^k$ one has
\begin{align*}
\Phi_n(A)D
&=\int_{\bT^k} t^{-n}\gamma_t(A)\, dt\, D\\
&=\int_{\bT^k} t^{-n}\gamma_t(A)\gamma_t(D)\, dt\quad (\text{as }\gamma_t(D)=D)\\
&=\int_{\bT^k} t^{-n}\gamma_t(AD)\, dt\\
&=\int_{\bT^k} t^{-n}\gamma_t(DA)\, dt\quad (\text{as }A\in \fD_\Lambda')\\
&=D\int_{\bT^k} t^{-n}\gamma_t(A)\, dt\quad (\text{as }\gamma_t(D)=D)\\
&=D\Phi_n(A).
\end{align*}
This proves $\Phi_n(A)\in\fD_\Lambda'$. 
\end{proof}

It turns out that, for $n\in\bZ^k$, any element in $\fD_\Lambda'$ of degree $n$ in a certain ``canonical" form is very 
special: It essentially has only one term. 
This is not surprising if one keeps the uniqueness result \cite[Lemma 4.1]{Yan14} in mind. 

\begin{lem}
\label{L:oneterm}
Let $m,n\in \bN^k$ and 
\[
\displaystyle A=\sum_{d(\mu)=m,\, d(\nu)=n, \, s(\mu)=s(\nu)} a_{\mu,\nu}\, s_\mu s_\nu^*\in \ca(\Lambda).
\]
Then the following hold true:
\begin{enumerate}
\item
 If $A\in\fD_\Lambda'$, 
then, for each $\mu\in\Lambda^m$, there is a unique $\nu\in\Lambda^n$ such that $a_{\mu,\nu}\ne 0$. 
\item
 If $A\in\fD_\Lambda'$, 
then, for each $\nu\in\Lambda^n$, there is a unique $\mu\in\Lambda^m$ such that $a_{\mu,\nu}\ne 0$. 
\end{enumerate}
\end{lem}

\begin{proof}
It suffices to show (i), since once (i) is established (ii) follows by applying (i) to $A^*$.

Let 
\[
\displaystyle A=\sum_{d(\mu)=m,\, d(\nu)=n,\, s(\mu)=s(\nu)} a_{\mu,\nu}\, s_\mu s_\nu^*\in \fD_\Lambda'.
\] 
Notice that $s_\mu s_\nu^*\ne 0$ as $s(\mu)=s(\nu)$.
Assume that $\mu_0\in\Lambda^m$ is such that $a_{\mu_0,\nu_0}\ne 0$ for some $\nu_0\in\Lambda^n$. Then we must show the uniqueness of $\nu_0$.

Since $A\in\fD_\Lambda'$, we have 
\[
s_{\mu_0}s_{\mu_0}^*A=A s_{\mu_0}s_{\mu_0}^*.
\]  
Multiplying $s_{\mu_0}^*$ from the left at both sides in the above identity induces
\[
s_{\mu_0}^*A=s_{\mu_0}^*As_{\mu_0}s_{\mu_0}^*
\]
as $s_{\mu_0}$ is a partial isometry.
Then we expand it using the formula of $A$ and then calculate both sides to obtain
\begin{align}
\label{E:uni}
\sum_{\nu\in \Lambda^ns(\mu_0)} a_{\mu_0,\nu}\, s_\nu^*
= \sum_{\nu\in \Lambda^ns(\mu_0)} a_{\mu_0,\nu}\, s_\nu^* s_{\mu_0}s_{\mu_0}^*.
\end{align}
Multiplying $s_{\nu_0}$ from right at both sides of \eqref{E:uni} and using (CK3) yields
\begin{align*}
&a_{\mu_0,\nu_0}s_{\nu_0}^*s_{\nu_0}= \sum_{\nu\in \Lambda^ns(\mu_0)} a_{\mu_0,\nu}\, s_\nu^* s_{\mu_0}s_{\mu_0}^* s_{\nu_0}\\
\Rightarrow\ &
a_{\mu_0,\nu_0}s_{\nu_0}\cdot s_{\nu_0}^*s_{\nu_0}\cdot s_{\nu_0}^*
= \sum_{\nu\in \Lambda^ns(\mu_0)} a_{\mu_0,\nu}\, s_{\nu_0}\cdot  s_\nu^* s_{\mu_0}s_{\mu_0}^* s_{\nu_0}\cdot s_{\nu_0}^*\\
\Rightarrow\ &
a_{\mu_0,\nu_0}s_{\nu_0}s_{\nu_0}^*
= \sum_{\nu\in \Lambda^ns(\mu_0)} a_{\mu_0,\nu}\, s_{\nu_0} s_\nu^* s_{\nu_0} s_{\nu_0}^*s_{\mu_0}s_{\mu_0}^*\quad (\text{as }P_{\mu_0} P_{\nu_0} =P_{\nu_0} P_{\mu_0})\\
\Rightarrow\ &
a_{\mu_0,\nu_0}s_{\nu_0}s_{\nu_0}^*
= a_{\mu_0,\nu_0}\, s_{\nu_0} s_{\nu_0}^*s_{\mu_0}s_{\mu_0}^*\\
\Rightarrow\ &
P_{\nu_0}=P_{\nu_0}P_{\mu_0}\quad(\text{by (CK3)} \text{ and }a_{\mu_0,\nu_0}\ne 0).
\end{align*}

Completely similar reasoning (by considering $A^*$ instead of $A$) gives $P_{\mu_0}=P_{\mu_0}P_{\nu_0}$.
Therefore we so far have shown that, for \textit{any} $\mu_0,\nu_0$ such that $a_{\mu_0,\nu_0}\ne 0$, we have
\begin{align}
\label{E:Pu=Pv}
P_{\mu_0}=P_{\nu_0}. 
\end{align}
From \eqref{E:uni} one also induces that
\begin{align*}
&\left(\sum_{\nu\in \Lambda^ns(\mu_0)} a_{\mu_0,\nu}\, s_\nu^*\right)  \left(\sum_{\nu\in \Lambda^ns(\mu_0)} a_{\mu_0,\nu}\, s_\nu^*\right)^*\\
=&\left(\sum_{\nu\in \Lambda^ns(\mu_0)} a_{\mu_0,\nu}\, s_\nu^* s_{\mu_0}s_{\mu_0}^*\right)\left( \sum_{\nu\in \Lambda^ns(\mu_0)} a_{\mu_0,\nu}\, s_\nu^* s_{\mu_0}s_{\mu_0}^*\right)^*.
\end{align*}
Expanding both sides and then using \eqref{E:Pu=Pv} and (CK3), we obtain
\begin{align*}
\sum_{\nu\in \Lambda^ns(\mu_0)} |a_{\mu_0,\nu}|^2 s_{s(\nu)}&=
\sum_{\nu_1,\,\nu_2\in \Lambda^ns(\mu_0)} a_{\mu_0,\nu_1}\ol a_{\mu_0,\nu_2}\, s_{\nu_1}^* s_{\mu_0}s_{\mu_0}^*   s_{\mu_0}s_{\mu_0}^* s_{\nu_2}\\
&=\sum_{\nu_1,\, \nu_2\in \Lambda^ns(\mu_0)} a_{\mu_0,\nu_1}\ol a_{\mu_0,\nu_2}\, s_{\nu_1}^*\cdot s_{\mu_0}s_{\mu_0}^* \cdot s_{\nu_2}\\
&=\sum_{\nu_1,\, \nu_2\in \Lambda^ns(\mu_0)} a_{\mu_0,\nu_1}\ol a_{\mu_0,\nu_2}\, s_{\nu_1}^*\cdot s_{\nu_0}s_{\nu_0}^*\cdot s_{\nu_2}\quad (\text{by }\eqref{E:Pu=Pv})\\
&= |a_{\mu_0,\nu_0}|^2 s_{\nu_0}^* s_{\nu_0}s_{\nu_0}^* s_{\nu_0}\quad(\text{from (CK3)})\\
&=|a_{\mu_0,\nu_0}|^2 s_{s(\nu_0)}.
\end{align*}
Therefore, $a_{\mu_0,\nu}=0$ for all $\nu\ne \nu_0$, proving the uniqueness of $\nu_0$. 
\end{proof}

We are now ready to prove our main result in this section. 

\begin{thm}
\label{T:MASA}
Let $\Lambda$ be a $k$-graph, and $\M_\Lambda$ be the cycline algebra of $\ca(\Lambda)$. 
Then $\M_\Lambda=\fD_\Lambda'$. In particular, $\M_\Lambda$ is a MASA in $\ca(\Lambda)$.
\end{thm}

\begin{proof}
We first prove $\M_\Lambda=\fD_\Lambda'$. The inclusion $\M_\Lambda\subseteq \fD_\Lambda'$ is clearly from Lemma \ref{L:equchar} (ii) $\Rightarrow$ (i).
So we must show $\fD_\Lambda'\subseteq \M_\Lambda$ in what follows.

It is easy to verify that $\fD_\Lambda'$ is a gauge invariant $\fD_\Lambda$-bimodule. 
Using an argument similar to \cite[Theorem 3.1]{HPP05} (also cf. \cite{Hop05}), one can show that $\fD_\Lambda'$ is 
generated by the standard generators $s_\mu s_\nu^*$ which it contains.
Let $\fA$ be the `algebraic' part of $\fD_\Lambda'$. That is, 
$\fA$ is the algebra of the \textit{finite} linear span of those standard generators. Then $\fA$ is dense in $\fD_\Lambda'$. 
So, for our purpose, it suffices to show that $\fA\subseteq \M_\Lambda$.
For this, let $A\in\fA$. 
By Lemma \ref{L:Phin}, without loss of generality, let us assume that $A$ is of degree $n$ for some $n\in \bZ^k$: 
$A=\Phi_n(A)$. Recall that $A$ is just a (finite) linear combination of some generators $s_\mu s_\nu^*$'s.
Using the defect-free property (CK4), one can now simply write $A$ as follows:
\[
\displaystyle A=\sum_{d(\mu)=m,\, d(\nu)=n,\, s(\mu)=s(\nu)} a_{\mu,\nu}\, s_\mu s_\nu^*,
\]
where $m,n$ are two fixed elements in $\bN^k$, and $a_{\mu,\nu}\ne 0$. 
It then follows from Lemma \ref{L:oneterm} 
for a fixed $\nu\in\Lambda^n$, there is a unique $\mu\in\Lambda^m$ such that $a_{\mu,\, \nu}\ne 0$.
Then 
\[
A s_\nu s_\nu^*=\left(\sum a_{\mu,\nu}\, s_\mu s_\nu^*\right)s_\nu s_\nu^*=a_{\mu,\nu} s_\mu s_\nu^*.
\]
Clearly $A s_\nu s_\nu^*\in\fD_\Lambda'$ as $A\in\fD_\Lambda'$ and $s_\nu s_\nu^*\in \fD_\Lambda$. So $s_\mu s_\nu^*\in\fD_\Lambda'$. 
By Lemma \ref{L:normal}, $\mu\sim \nu$, which implies that $A\in \M_\Lambda$. Therefore, $\M_\Lambda=\fD_\Lambda'$. 

The second part of the theorem follows immediately, since it is known that $\fD_\Lambda'$ is abelian and $\M_\Lambda'=\fD_\Lambda'$ (\cite[Proposition 7.3]{BNR14}).
\end{proof}

By Theorem \ref{T:MASA}, one can easily recover the following characterization of aperiodicity, one of the main theorems 
in \cite{Hop05}. 

\begin{cor}
\label{C:masa}
A $k$-graph $\Lambda$ is aperiodic if and only if the diagonal algebra $\fD_\Lambda$ is a MASA in $\ca(\Lambda)$. 
\end{cor}

\begin{proof}
$\fD_\Lambda$ is a MASA in $\ca(\Lambda)$, if and only if $\fD_\Lambda'=\fD_\Lambda$, if and only if $\M_\Lambda=\fD_\Lambda$ by Theorem \ref{T:MASA}, 
if and only if no $\mu\ne \nu$ such that $\mu\sim \nu$ by definition of $\M_\Lambda$, if and only if $\Per\Lambda=\{0\}$ by definition of $\Per\Lambda$,
if and only if $\Lambda$ is aperiodic  (see, e.g., \cite{RS07} or \cite{Yan14}).  
\end{proof}


\section{When are Cyline subalgebras Cartan}

\label{S:Cartan}

Let $\B$ be an abelian C*-subalgebra of a given C*-algebra $\A$. 
Recall from \cite{Ren08} that $\B$ is a \textit{Cartan subalgebra} in $\A$ if the following properties hold:
\begin{itemize}
\item[(Ci)] $\B$ contains an approximate unit in $\A$;
\item[(Cii)] $\B$ is a MASA;
\item[(Ciii)] $\B$ is regular, i.e., the normalizer set $N(\B)=\{x\in \A: x\B x^*\cup x^*\B x\subseteq \B\}$ generates $\A$;
\item[(Civ)] there is a faithful conditional expectation $\E$ from $\A$ onto $\B$. 
\end{itemize}

Let $\Lambda$ be a $k$-graph. 
In this section, we prove that the cycline subalgebra $\M_\Lambda$ of $\ca(\Lambda)$ is Cartan under 
the condition that the (bimodule) spectrum of $\M_\Lambda$ (the definition will be given later) is closed, 
which is used to obtain property (Civ). 

\begin{prop}
\label{P:regular}
Let $\Lambda$ be a $k$-graph, and $\M_\Lambda$ be the cycline subalgebra of $\ca(\Lambda)$. 
Then $\M_\Lambda$ is regular. 
\end{prop}

\begin{proof}
Since $\ca(\Lambda)$ is generated by its standard generators $s_\alpha s_\beta^*$'s $(\alpha,\beta\in\Lambda$), 
it suffices to show that every $s_\alpha s_\beta^*$ is a normalizer of $\M_\Lambda$. 
But $\M_\Lambda$ is generated by $s_\mu s_\nu^*$ with $\mu\sim\nu$. Hence, one only needs to show that 
\[
s_\alpha s_\beta^*  s_\mu s_\nu^*  s_\beta s_\alpha^*\in\M_\Lambda.
\]
To this end, let us assume that 
\[
 s_\beta^* s_\mu=\sum_{\beta \mu'=\mu \beta'} s_{\mu'} s_{\beta'}^*\quad\text{and}\quad
 s_\nu^* s_\beta=\sum_{\beta \nu'=\nu \beta''} s_{\beta''} s_{\nu'}^*.
\] 
Then 
\begin{align}
\nonumber
&s_\alpha s_\beta^* s_\mu s_\nu^* s_\beta s_\alpha^*\\
\nonumber
=&s_\alpha\left(\sum_{\beta \mu'=\mu \beta'} s_{\mu'} s_{\beta'}^*\right)\left(\sum_{\beta \nu'=\nu \beta''} s_{\beta''} s_{\nu'}^*\right)s_\alpha^*\\
\nonumber
=&s_\alpha\left(\sum_{\beta \mu'=\mu \beta', \, \beta \nu'=\nu \beta'} s_{\mu'} s_{\nu'}^*\right)s_\alpha^*
\quad \big(\text{as }d(\beta')=d(\beta'') \Rightarrow s_{\beta'}^* s_{\beta''}=\delta_{\beta',\beta''}\, s_{s(\beta')}\big)\\
\label{E:end}
=&\sum_{\beta \mu'=\mu \beta', \, \beta \nu'=\nu \beta'} s_{\alpha\mu'} s_{\alpha\nu'}^*.
\end{align}
Clearly $\mu\beta'\sim \nu\beta'$ as $\mu\sim\nu$. It follows from $\beta \mu'=\mu \beta', \, \beta \nu'=\nu \beta'$ that $\beta\mu'\sim\beta\nu'$.
This easily implies $\mu'\sim\nu'$ and so $\alpha\mu'\sim\alpha\nu'$. 
Therefore one has $s_\alpha s_\beta^* s_\mu s_\nu^* s_\beta s_\alpha^*\in\M_\Lambda$ from \eqref{E:end}. 
This ends our proof. 
\end{proof}

In what follows, we identify $\ca(\Lambda)$ with $\ca(\G_\Lambda)$ under the isomorphism mapping $s_\alpha s_\beta^*\mapsto 1_{Z(\alpha,\beta)}$,
where $1_{Z(\alpha,\beta)}$ is the characteristic function of the cylinder set $Z(\alpha,\beta)$ (cf. \cite{KumPask}). Then the diagonal algebra $\fD_\Lambda$ 
is identified as $ \rC_0(\G_\Lambda^{(0)})$,
where $\G_\Lambda^{(0)}$ is the unit space of $\G_\Lambda$;
and $\M_\Lambda$ is generated by $1_{Z(\mu,\nu)}$'s with $\mu\sim \nu\in\Lambda$.
Evidently, $\M_\Lambda$ is a $\fD_\Lambda$-bimodule. By definition \cite{Hop05}, its (bimodule) \textit{spectrum} is 
\[
\sigma(\M_\Lambda)=\big\{(x,n,y)\in \G_\Lambda: f(x,n,y)\ne 0\text{ for some }f\in \M_\Lambda\big\}. 
\]
Clearly, $\sigma(\M_\Lambda)$ is always an open subset of $\G_\Lambda$.

\begin{prop}
\label{P:expectation}
Let $\Lambda$ be a $k$-graph, and $\M_\Lambda$ be the cycline subalgebra of $\ca(\Lambda)$. 
If $\sigma(\M_\Lambda)$ is closed in $\G_\Lambda$, then there is a faithful conditional expectation from $\ca(\Lambda)$ onto $\M_\Lambda$.
\end{prop}

\begin{proof}
It is easy to see that $\M_\Lambda$ is a gauge invariant $\fD_\Lambda$-bimodule. 
It then follows from \cite[Spectral Theorem for Bimodules on P.~997]{Hop05} that one has
\[
\M_\Lambda=\big\{f\in \ca(\G_\Lambda): f(x,n,y)=0 \text{ for all } (x,n,y)\not\in \sigma(\M_\Lambda)\big\}.
\]
Since $\sigma(\M_\Lambda)$ is closed in $\G_\Lambda$, one now can define the restriction mapping 
$\E: \ca(\G_\Lambda)\to \M_\Lambda$ via $\E(f)=f|_{\sigma(\M_\Lambda)}$ for all $f\in\ca(\G_\Lambda)$. 
Clearly, $\E$ is a linear idempotent. Using the definition of the norm in 
$\ca_r(\G_\Lambda),$\footnote{Recall that the reduced C*-norm on $\rC_c(\G_\Lambda)$ is 
given by 
\[
\|f\|=\sup\big\{\|\lambda_u(f)\|:u\in\G_\Lambda^{(0)}\big\} \text{ for all }f\in \rC_c(\G_\Lambda),
\] 
where $\lambda_u$ is the regular representation of $\rC_c(\G_\Lambda)$ on $\ell^2(s^{-1}(u))$: 
\[
\lambda_u(f)\xi(\gamma)=\sum_{\alpha\beta=\gamma}f(\alpha)\xi(\beta)\text{ for all } u\in \G^{(0)},\, \xi\in \ell^2(s^{-1}(u))\text{ and }\gamma\in s^{-1}(u).
\]
} 
one can see that $\E$ is also contractive.
Furthermore, it is known that  the restriction mapping $E$ from $\ca(\G_\Lambda)$ to $\fD_\Lambda$  yields a faithful conditional expectation onto $\fD_\Lambda$ 
(see, e.g., \cite{BCS14, Ren08, Tho10}). In particular, $\|E\|=1$. 
Then it follows from $E=E\circ \E$ that $\E$ has norm $1$. Thus $\E$ is a conditional expectation onto $\M_\Lambda$ (\cite{Tom57} or \cite[II.6.10]{Bla06}). 
Moreover, the faithfulness of $E$ and $E=E\circ \E$ imply that $\E$ is faithful too. 
This ends the proof.
\end{proof}

Let us remark that the above proposition could be also proved by modifying the proof of \cite[Lemma 2.21]{Tho10}. 
Notice that the mapping $Q$ in (2.16) there is our restriction mapping. 
As our condition of $\sigma(\M_\Lambda)$ being closed, the discreteness of the isotropy group guarantees that $Q$ is well-defined.
Also, one has $\M_\Lambda=\ca_r(\sigma(\M_\Lambda))$ by \cite[Lemma2.10]{Tho10}.

\begin{thm}
\label{T:Cartan}
Let $\Lambda$ be a $k$-graph. Suppose that the (bimodule) spectrum $\sigma(\M_\Lambda)$ is closed in $\G_\Lambda$.  
Then $\M_\Lambda$ is a Cartan subalgebra in $\ca(\Lambda)$. 
\end{thm}
\begin{proof}
The proof is now immediate: 
Property (Ci) follows from \cite[Lemma 2.1]{BCS14} (also cf.~the proof of \cite[Theorem 2.23]{Tho10});
properties (Cii), (Ciii) and (Civ) are from 
Theorem \ref{T:MASA}, Proposition \ref{P:regular} and Proposition \ref{P:expectation}, respectively. 
\end{proof}

\begin{rem}
It is probably worthwhile to mention that Q1 and Q2 mentioned in Introduction were successfully attacked for directed graphs (i.e., $1$-graphs) in \cite{NR12}.
Unfortunately, the methods there can not be applied to $k$-graphs, due to the complexity caused by periodicity 
in higher dimensional cases. Notice that Theorem \ref{T:Cartan} is proved in \cite{Yan14} for a special class of $k$-graphs as an application of an embedding theorem. 
\end{rem}

\subsection*{Acknowledgements} I would like to thank Prof. Ken Davidson and Prof. Laurent Marcoux for inviting me to 
give a talk on this paper at the University of Waterloo. 
Special thanks go to Prof. Aidan Sims for pointing out an error in the first version of this paper. 
The author is also very grateful for the referee's careful reading and useful comments. 

\subsection*{Note added in proof} After this paper was circulated, the main results of this paper were generalized by Brown-Nagy-Reznikoff-Sims-Williams in
the recent paper \cite{BNRSW15} by using completely different approaches. 


\end{document}